\documentclass[12pt]{amsart}

\usepackage{graphicx}
\usepackage[iso-8859-7]{inputenc}
\usepackage{amsthm}
\theoremstyle{plain}
\newtheorem*{theorem*}{Theorem}
\usepackage{amsmath}
\usepackage{xcolor}
\usepackage{verbatim}

\newtheorem{theorem}{Theorem}
\newtheorem*{Theorem A}{Theorem A}
\newtheorem*{th*}{Theorem}
\newtheorem{lemma}{Lemma}
\newtheorem*{Lemma A}{Lemma A}
\newtheorem*{Lemma B}{Lemma B}
\newtheorem*{Lemma C}{Lemma C}
\newtheorem*{Lemma D}{Lemma D}
\newtheorem*{Lemma E}{Lemma E}
\newtheorem*{Lemma F}{Lemma F}
\newtheorem*{lemma*}{Lemma}

\theoremstyle{definition}
\newtheorem{definition}{Definition}

\newtheorem{proposition}{Proposition}
\theoremstyle{remark}

\numberwithin{equation}{section}

\def\p{\psi}

\def\D{{\mathbb D}}

\def\({\left(}       \def\){\right)}

\usepackage{enumerate}
\numberwithin{equation}{section}
\usepackage{mathtools}
\usepackage{amsmath,amssymb}

\begin{document}

\title[Schatten Class Hankel Operators]{Schatten Class Hankel Operators on Large Bergman spaces}
\author{P. Galanopoulos}
\address{Dept. of Mathematics, Aristotle University of Thessaloniki, 54124, Thessaloniki, Greece }
\email{petrosgala@math.auth.gr}


\begin{abstract}
We  characterize the membership  of Hankel Operators with general symbols in the Schatten Classes  $S^p,\, p\in(0,1),$
of the large Bergman spaces $A^2_{\omega}$. The case $p\geq 1$ was proved by Lin and Rochberg.
\end{abstract}

\maketitle

\section{Introduction}
Let $\D$ be the unit disc of the complex plane,
$H(\D)$ be the space of analytic functions in the disc and $dA(z)=\frac{dxdy}{\pi}$ be the normalized area Lebesque measure in $\mathbb D$.
We say that an $f\in H(\mathbb D)$ belongs to the Hilbert space $A^2_{\omega}$
if
\begin{equation}\label{norm}
\|f\|^2_{A^2_{\omega}}= \int_{\mathbb D}\, |f(z)|^2 \,\omega(z) \, dA(z) <\infty\,.
\end{equation}
The weight $\omega$ is positive, radial, integrable in the disc and belongs to  the class $\mathcal W$.
Among the members of $\mathcal W$ we can find those weights that introduce the  
typical large Bergman spaces
like the exponential
$$
\omega(z)= e^{-\frac{1}{(1-|z|)^\alpha}}\,, \quad \alpha >0\,,
$$
and  the double exponetial weights
$$
\omega(z)= e^{-\gamma e^{\frac{\beta}{(1-|z|)^\alpha}}}\,, \quad \alpha,\beta,\gamma >0\,.
$$
The class $\mathcal W$ was originally considered in \cite{BDK},\,\cite{PP}.
We will present all the necessary details in the next section.

Let  $L^2_{\omega}$  be the space of the measurable functions that satisfy (\ref{norm}).
If $P$ is the projection of $L^2_{\omega}$ onto $A^2_{\omega}$ then 
we call Hankel operator with symbol a function 
$g\in L^2_{\omega}$ the operator 
\begin{equation}\label{hankel}
H_g (f) = g f - P (gf) \,,
\end{equation}
defined on a dense subset of $ A^2_{\omega}$\,.
The action of Hankel operators on the standard weighted Bergman spaces $A^2_{\alpha}$, that is for 
$$
\omega(z)=(1-|z|)^{\alpha}\,,\quad\alpha>-1,
$$ 
has been extensively studied. We refer to 
\cite{A}, \cite{AFP}, \cite{FX}, 
\cite{Lu}, \cite{P},
\cite{X}, 
 \cite{Zh1}, \cite{Zh2} for more information. 
In \cite{LR1} Lin and Rochberg considered the action of Hankel operators  
on large Bergman spaces and they characterized the boundedness and the compactness.
In a following work, among other things, they dealt with the membership of $H_g$ in the Schatten classes  $S^p=S^p(A^2_{\omega})$ for $p\geq 1$. See \cite{LR2}.
The weights in \cite{LR1}, \cite{LR2} were originally introduced in \cite{O}, \cite{OP} and are slightly different from those of the class $\mathcal W$.
The action of Hankel operators with conjugate analytic symbols on $A^2_{\omega},\, \omega \in \mathcal W,$ was considered in \cite{GP}.

The members of $\mathcal W$ are radial weights. As a consequence, like in the case of the unweighted Bergman spaces,
 we can prove that the polynomials are dense in $A^2_{\omega}$.  Note that this does not necessarily hold for non radial weights.
See Remark $4.3$ in \cite{LR2}. Therefore, a Hankel operator on  $A^2_{\omega}$ with symbol $g \in L^2_{\omega}$ is the densely defined operator
$$
H_g(f)= g\,f \,-\,P(g\,f)\,, \quad \quad f\in H^{\infty}\,.
$$
Here $H^{\infty}$ is the algebra of  bounded analytic functions in $\mathbb D$. 

The objective of this work is the study of the Schatten classes.
As far as it concerns the boundedness of $H_g$  we refer to the proof  of Theorem $3.1$ in \cite{LR1}. Under minor modifications it serves in the case of $\omega \in \mathcal W$
as well. We state the characterization of boundedness postponing all the requirements for the next section. 
\begin{Theorem A}
Let $\omega \in \mathcal W$ and  $g\in L^2_{\omega}$. The following are equivalent\\
$(1)  \,\,\,\,H_g \in  B( A^2_{\omega},  L^2_{\omega})$\\
$(2)$ The function 
$$
 F^2_{\delta}(z)=\inf\left\{\frac{1}{|D(\delta \tau (z))|}\,\int_{ D(\delta \tau (z))}\, |g - h|^2 \, dA : h\,\, \text{analytic in }  D(\delta \tau (z))  \right\}
$$
is bounded, that is  $ \sup_{z \in \mathbb  D} F_{\delta}(z) < \infty  $
 for some  $\delta \in (0,m_{\tau})$.\\
 $(3)$ The symbol $g$ admits a decomposition  $g=g_1 + g_2$ where   $g_2 \in {\mathcal C}^1(\mathbb D)$ and 
 \begin{equation}\label{bar}
\frac{\bar{\partial}g_2}{(\Delta \phi)^{\frac 12}}\in L^{\infty}(\mathbb D)
\end{equation}
 while $g_1$ satisfies the following condition 
 \begin{equation}\label{CM}
\sup_{z \in \mathbb  D}\, \frac{1}{|D(\delta\tau(z))|}\int_{D(\delta\tau(z))}\,|g_1|^2\, dA\,<\infty
\end{equation}
for some $\delta \in (0,m_{\tau})$.
\end{Theorem A}
 
Let $\mathcal H$ be a Hilbert space and T be a bounded operator on $\mathcal H$. We say that $T$ belongs to the Schatten class $S^p$ of $\mathcal H$,\,$ p\in (0,\infty),$ if its 
 sequence of  singular numbers $\{s_\nu(T)\}_\nu \in l^p\,.$
The singular numbers of $T$ are defined as
$
s_\nu=s_\nu (T)=\inf\{ \|T-K\| : rank K \leq \nu \}.
$
Let $ \| T \|^{p}_{S^p} =\sum_{n} s_{\nu}^p\,. $ If $p\geq 1$ then it is a norm while for $p \in (0,1)$ it is true that
$
\| T + S \|^{p}_{S^p} \leq \| T \|^{p}_{S^p} + \| S \|^{p}_{S^p}\,.
$
We suggest \cite{GK}, \cite{Zh3} for more information on the theory of the Schatten Ideals. 
Although we work with $\omega\in\mathcal W$ the membership of Hankel operators in $S^p$ of $A^2_{\omega}$ for $p\geq 1$ is actually characterized by Theorem $4.4'$ \cite{LR2}. We deal with the open case $p\in(0,1).$
\begin{theorem}\label{main}
Let $\omega \in \mathcal W$ and  $g\in L^2_{\omega}$. Assume that 
$H_g$ is bounded in the  $L^2_{\omega}$ norm. The following are equivalent\\
$(1)  \,\,\,\,H_g \in S^p,\,\quad p\in (0,1)$\\
$(2)$ There exists a $\delta \in (0,m_{\tau})$  such that for every $\tau- covering$  $\{ D(\delta \tau(z_j))\}_j$  of the unit disc 
$$ \sum_j F_{\delta}(z_j)^p < \infty   $$
$(3)$  There exists a $\delta \in (0,m_{\tau})$ and a decomposition  $g=g_1 + g_2$ 
such that for every $\tau-$ covering  $\{ D(\delta \tau(z_j))\}_j$  of $\mathbb D$
$$
\sum_j \,\left( \frac{1}{|D(\delta\tau(z_j))|}\int_{D(\delta\tau(z_j))}\,|g_1|^2\, dA(z)\right)^{p/2}\,<\infty
$$
while $g_2 \in {\mathcal C}^1(\mathbb D)$ and the same holds with 
$
\frac{\bar{\partial}g_2}{(\Delta \phi)^{\frac 12}}
$
in place of $g_1$.
\end{theorem}
The implications $(2) \Rightarrow (3)$  and $(3) \Rightarrow (1)$ turn out only
by applying the corresponding steps of the proof of Theorem $4.4'$.  We 
present the arguments for the sake of completeness. Our contribution is the proof of $(1) \Rightarrow (2).$ 
Throughout the paper $A\lesssim B $ means that there is a constant $C$ independent
of the relevant variables such that $A \leq C B.$  We write $A \asymp B$ if both $A \lesssim B$ and $B \lesssim A $. We follow the notation introduced by Lin and Rochberg in \cite{LR1}, \cite{LR2}.

\section{Preliminaries}
\subsection{The class $\mathcal W$} We will follow the notation introduced in \cite{PP}, \cite{APP}.
We say that a positive function $\tau$ in $\mathbb D$ belongs to the class $\mathcal L$ if
\begin{equation}\label{c1}
\exists\,\,  c_1>0 \,\,\,\, \text{such that} \,\,\,\, \tau(z) \leq \,c_1  (1-|z|)\,,\,\, \forall\,\,z \in \mathbb D \\
\end{equation}
\begin{equation}\label{c2}
 \exists\,\,  c_2>0 \,\,\,\, \text{such that}\,\, \,\,|\tau(z)-\tau(\zeta)|\leq\, c_2\, |z - \zeta|\,,\,\, \forall\,\,z ,\zeta \in \mathbb D \,.
\end{equation}
Let $\delta > 0$ and $z\in \mathbb D$. We denote by $D(\delta \tau (z))$ the Euclidean disk with center the point $z$ and radius $\delta\tau (z)$.
According to Lemma 2.1 in \cite{PP} if we choose a $\delta \in (0, m_{\tau}),$
 where $m_{\tau} = \frac{min(1,c^{-1}_{1},c^{-1}_{2})}{4},$
   and a $z \in \mathbb D$
 then for any  $\zeta\in \overline{D(\delta\tau(z))}$  
\begin{equation}\label{comparison tau}
\frac 12 \tau(z) \leq \tau(\zeta) \leq 2 \tau(z)
\end{equation}
In \cite{O} it is established the following covering lemma for the unit disk.
\begin{Lemma A}
If $\tau \in {\mathcal L}$ and $\delta \in (0,m_{\tau})$  then there exists a sequence of points $\{ z_j \}_j \subset \mathbb D,$
such that the following conditions are satisfied\\
$(1)$ $z_j \notin D(\delta\tau(z_k)),\quad j\neq k$\\
$(2)$ $\cup_j  D(\delta\tau(z_j))= \mathbb D $\\
$(3)$ $D(\delta\tau(z_j))=\cup_{z\in D(\delta\tau(z_j))}\, D(\delta\tau(z))\subset D(3\delta\tau(z)),\quad j=1,2,...$\\
$(4)$ $ \left\{  D(3\delta\tau(z_j))\right\}$ is a covering of $\mathbb D$ of finite multiplicity $N$, which is independent of $\delta$.
\end{Lemma A}
This collection $\{  D(\delta\tau(z_j))\}_j$ is called a $\tau$-covering of $\mathbb D$ and the sequence of points $\{ z_j\}_j$ a $(\delta,\tau)$-lattice of $\mathbb D$.

\begin{definition}
 We say that a positive, integrable function
$
\omega
$
 in $\mathbb D$
belongs to the class $\mathcal W$  if 
$$
 \omega (z)=e^{-2\phi (z)}\,, \quad z\in \mathbb D,
$$ 
where $\phi $ is a radial function such that 
$$ \phi \in C^2(\mathbb D)\quad \text{and }\quad(\Delta \phi(z))^{-\frac 12} \asymp \tau(z) $$
where the function $\tau$  is a positive radial function that decreases to $0$ as $|z|\to 1^{-}$ and $\lim_{r \to 1^{-}} \tau'(r) =0$. 
Here $\Delta $ stands for the Laplace operator.
We also supose that either there exist a constant $C>0$ such that $\tau (r)(1-r)^C$
increases for $r$ close to $1$ or $\lim_{r \to 1} \tau'(r) \log\frac{1}{\tau(r)}=0.$
\end{definition}

It is straightforward to see that if $\omega \in \mathcal W $ then $\tau \in \mathcal L$.  
Moreover,  any  $(\delta, \tau)-$lattice $\{ z_j\}_j$ can be partitioned into a finite number of subsequences such that any two distinct points of the same subsequence  are sufficiently far one from the other. To be more specific, let 
 \begin{equation}\label{distance}
 d_{\tau}(z,\zeta)=\frac {|z-\zeta|}{min(\tau(z),\tau(\zeta))}\,,\quad z,\zeta\in \mathbb D\,.
 \end{equation}
 In \cite{APP} Lemma $4.5$  the authors prove the following argument.
 \begin{Lemma B}
 If $\omega \in \mathcal W$, $\delta \in (0,m_{\tau})$ and $k$ is a positive integer then any $(\delta, \tau)-$lattice $\{ z_j\}_j $ of $\mathbb D$
can be partitioned into $M$ subsequences such that, if $z_j$ and $z_i$ are different points in the same subsequence then
 $d_{\tau}(z_i, z_k)\geq 2^k\delta\,.$
  \end{Lemma B}
\subsection{Bergman spaces and test functions} 
The space $A^2_{\omega}$, $\omega \in \mathcal W ,$  is a  Hilbert space
with inner product 
$$
<f,g>=\int_{\mathbb D} f(z)\,\overline{g(z)}\, \omega(z)\,dA(z)\,.
$$
In Lemma $2.2$ \cite{PP} it is proved that the point evaluation functionals $L_z$ at a $z\in \mathbb D$ are bounded. Therefore, there are reproducing kernel functions $K_z \in A^2_{\omega}$
with 
$\|L_z\|=\|K_z\|_{A^2_{\omega}}$ 
such that 
$
L_z(f)=<f, K_z>=f(z)\,.
$
In many cases the normalized reproducing kernels are used as test functions.
Here, we are interested in the following case
\begin{Lemma C}
Let $n\in \mathbb N \backslash\{0\}$ and $\omega \in \mathcal W$. There is a number $\rho_0\in (0,1)$
 such that for each $\alpha \in \mathbb D$
 with $|\alpha|\geq\rho_0$  there is a function  $F_{\alpha, n} \in H^{\infty}$ 
 such that,  uniformly in $\alpha$,
 \begin{equation}\label{test function1}
 |F_{\alpha, n}(z)|^{2}\,\omega(z) \asymp 1\,,\quad \text{if } \quad |z-\alpha|<\tau(\alpha)
 \end{equation}
 and 
 \begin{equation}\label{test function2}
 |F_{\alpha, n}(z)|\,\omega(z)^{\frac 12} \lesssim \,\min  \left( 1, \frac{\min (\tau(\alpha), \tau(z))}{|z-\alpha|} \right)^{3n}\,,\quad  z\in \mathbb D\,.
 \end{equation}
 Moreover,  it is true that
 \begin{equation}\label{norm test}
 \|F_{\alpha, n}\|_{A^{2}_\omega} \asymp \tau(\alpha)\,\quad \rho_0\leq|\alpha|<1\,.
  \end{equation}
 \end{Lemma C}
For the case $n=1$ see Lemma $3.3$ \cite{BDK} and for the other values of $n$ look at Lemma $3.1$ \cite{PP}.
 
 \subsection{Local Bergman Spaces} Let $\alpha \in \mathbb D, \delta > 0$. Consider the space
$L^2(D(\delta\,\tau(\alpha)),\omega\,dA)\doteqdot L^2_{\omega}(D(\delta\,\tau(\alpha))),$ 
the closed subspace of analytic functions $A^2_{\omega}(D(\delta\,\tau(\alpha)))$
and the projection $P_{\alpha, \delta}$ of $L^2_{\omega}(D(\delta\,\tau(\alpha)))$ onto $A^2_{\omega}(D(\delta\,\tau(\alpha)))$.
Given an $f \in L^2_{\omega}(D(\delta\,\tau(\alpha)))$ we extend $P_{\alpha, \delta}(f)$ to $\mathbb D$ by setting 
$
P_{\alpha, \delta}(f)|_{\mathbb D \setminus D(\delta\,\tau(\alpha) )}=0 \,.
$
It is true that 
$
P^2_{\alpha, \delta}(f)=P_{\alpha, \delta}(f)$ and $ <f,P_{\alpha, \delta}(g)>=<P_{\alpha, \delta}(f),g>$ for
$f,g \in L^2_{\omega}.
$
Moreover, if $ G \in A^2_{\omega}$ then 
$$
P_{\alpha,\delta}(G)=\chi_{D(\delta\tau(\alpha))}\,G
\quad \text{and} \quad <G\,,\,\chi_{D(\delta\tau(\alpha))}\,g-P_{\alpha,\delta}(g) >=0\,.$$
It is true that 
 \begin{lemma}\label{local bergman}
 	Let $f,g \in L^2_{\omega} $ then 
 	$$
 	<f-P(f),\, \chi_{D(\delta\,\tau(\alpha))}g - P_{\alpha, \delta}(g)> = <\chi_{D(\delta\,\tau(\alpha))}f-P_{\alpha, \delta}(f),\, \chi_{D(\delta\,\tau(\alpha))}g - P_{\alpha, \delta}(g)  >
 	$$
 \end{lemma}	
 \begin{proof}
 \begin{align*}
& <f-P(f),\, \chi_{D(\delta\,\tau(\alpha))}g - P_{\alpha, \delta}(g)> 
 =<f,\, \chi_{D(\delta\,\tau(\alpha))}g - P_{\alpha, \delta}(g)> \\
 = &< f-P_{\alpha, \delta}(f),\, \chi_{D(\delta\,\tau(\alpha))}g - P_{\alpha, \delta}(g)>
 = < \chi_{D(\delta\,\tau(\alpha))}\,f-P_{\alpha, \delta}(f),\, \chi_{D(\delta\,\tau(\alpha))}g - P_{\alpha, \delta}(g)>\,.
 \end{align*}
 \end{proof}
The following is based on the ideas of Proposition $2$ \cite{PP}.
\begin{proposition}\label{operatorA}
Let $\omega \in \mathcal W,\, \delta\in (0,m_{\tau}).$ Consider the lattice
from Lemma $A$, an orthonormal sequence  $\{e_j\}_j$ of $A^2_{\omega}$ and an $n\in \mathbb N$ such that $n\geq 2$. Let
 $\rho_0$ be the number given in Lemma C and $\{z_j\}_j $ be an enumeration of the points of the lattice with the property $|z_j|\geq \rho_0$. 
 If $J$ is any finite subcollection then the operator  
$$A(f)=\sum_{j\in J}\, <f,e_j>\kappa_j \,, \quad \kappa_j(z)=\frac{F_{z_j, n}(z)}{\|F_{z_j, n} \|_{A^{2}_{\omega}}}$$
 is bounded in $A^{2}_{\omega}.$ 
\end{proposition}	
\begin{proof}
	Observe that
	 $$	\|A(f)\|^2_{A^{2}_{\omega}}
		\leq \int_{\mathbb D}\, \left(\sum_{|z_j|\geq \rho_0}\, |<f,e_j>||\kappa_j(z)| \right)^2 \omega(z)dA(z)\,.$$
		 We write
		 $\sum_{|z_j|\geq \rho_0} |<f,e_j>||\kappa_j(z)|=\sum_{j} |<f,e_j>||\kappa_j(z)|\,.$
Applying  Cauchy Schwartz 
\begin{align*}
	\|A(f)\|^2_{A^{2}_{\omega}}
	& \lesssim  \int_{\mathbb D}\, \left(\sum_{j}\,| <f,e_j>|^2 \frac{|F_{z_j,n}(z)|^{2\frac{n-1}{n}}}{\tau^4(z_j)}\right)
	\,\left(\sum_{j}\,\tau^2(z_j) |F_{z_j,n}(z)|^{\frac{2}{n}}\right)\, \omega(z) \, dA(z)\,.
\end{align*}
We estimate the second factor.  
Let $z\in \mathbb D$ then employing (\ref{test function1}), (\ref{comparison tau}) and 
$(iv)$ of Lemma A
\begin{align}\label{first additive}
\sum_{\{z_j\in D(\delta_0\tau(z))\}}\,\tau^2(z_j) |F_{z_j,n}(z)|^{\frac{2}{n}}
\lesssim \frac{\tau^2(z)}{\omega^{\frac 1n}(z)}\,\,.
\end{align}
On the other hand, applying (\ref{test function2})
\begin{align*}
 \sum_{\{z_j\notin D(\delta_0\tau(z))\}}\,\tau^2(z_j) |F_{z_j,n}(z)|^{\frac{2}{n}}
 &\lesssim \frac{\tau^6(z)}{\omega^{\frac 1n}(z)}\sum_{i\geq 0}
 \sum_{z_j \in R_i(z)}\frac{\tau^{2}(z_j)}{|z-z_j|^6}\\
 &\lesssim \,\frac{1}{\omega^{\frac 1n}(z)}
 \,\sum_{i\geq 0} 2^{-6i}
 \sum_{z_j \in R_i(z)}\tau^{2}(z_j)
 \end{align*}
where $R_i(z)=\{\zeta\in \mathbb D : 2^i \delta_0 \tau(z) \leq |\zeta -z| < 2^{i+1} \delta_0 \tau(z) \},\, i=0,1,2...$
Due to (\ref{c2}), 
$$D(\delta_0\tau(z_j))\subset D(5\delta_0 2^{i}\tau(z))\,,\quad  \text{if} \,\, z_j\in D(\delta_0\, 2^{i+1}\tau(z)) $$ so
$$ \sum_{z_j \in R_i(z)}\tau^{2}(z_j) \lesssim A(D(z, 5\delta_0 2^{i}\tau(z)))\lesssim 2^{2i}\tau^2(z)\,.$$
This together with the finite multiplicity of the covering imply that
\begin{align}\label{second additive}
	\sum_{\{z_j\notin D(\delta_0\tau(z))\}}\,\tau^2(z_j) |F_{z_j,n}(z)|^{\frac{2}{n}}
&\lesssim \,\frac{\tau^2(z)}{\omega^{\frac 1n}(z)}\,.	
\end{align}
Combining (\ref{first additive}), (\ref{second additive}) we get that
\begin{align*}
	\sum_{j}\,\tau^2(z_j) |F_{z_j,n}(z)|^{\frac{2}{n}} \lesssim \frac{\tau^2(z)}{\omega^{\frac 1n}(z)}\,, \quad z\in \mathbb D.
\end{align*}
 Consequently,  
\begin{align}\label{operator A}
	\|A(f)\|^2_{A^{2}_{\omega}}\,
	& \lesssim
	\sum_{j}\,\frac{| <f,e_j>|^2}{\tau^4(z_j)}\,\int_{\mathbb D}\, |F_{z_j,n}(z)|^{2\frac{n-1}{n}}  \tau^2(z)\,\omega^{\frac{n-1}{n}}(z)\,dA(z)\,.
\end{align}
Let an index $j$ and consider the decomposition 
\begin{align*}
	\int_{\mathbb D}\, |F_{z_j,n}(z)|^{2\frac{n-1}{n}}  \tau^2(z)\,\omega^{\frac{n-1}{n}}(z)\,dA(z)=\int_{\{|z-z_j|<\tau(z_j)\}}\,\,\, + \int_{\{|z-z_j|\geq \tau(z_j)\}}= I_1 + I_2 \,.
\end{align*}
It follows from  (\ref{test function1}) that
\begin{align}\label{additive A}
I_1=\int_{\{|z-z_j|<\tau(z_j)\}}\, |F_{z_j,n}(z)|^{2\frac{n-1}{n}}  \tau^2(z)\,\omega^{\frac{n-1}{n}}(z)\,dA(z) \asymp  \tau^2(z_j)\,.
\end{align}
About the other additive, property (\ref{c2}) implies that $\tau(z)\lesssim 2^i \tau(z_j)$ when $|z-z_j|<2^i \tau(z_j).$
Since $n\geq 2,$ bearing in mind (\ref{test function2}),
\begin{align*}&I_2=\int_{\{|z-z_j|\geq \tau(z_j)\}}|F_{z_j,n}(z)|^{2\frac{n-1}{n}}\,  \tau^2(z)\,\omega^{\frac{n-1}{n}}(z)\,dA(z)\\
&\lesssim \tau^{6(n-1)}(z_j) \sum_{i\geq 0}\int_{\{2^{i}\tau(z_j)\leq |z-z_j|< 2^{i+1}\tau(z_j)\}}\,\frac{\tau^2(z)}{|z-z_j|^{6(n-1)}}\,dA(z) 
&\lesssim \tau^{4}(z_j)\,.
\end{align*}
The latter estimate together with (\ref{additive A}), (\ref{operator A}) and Bessel's inequality give	
\begin{align*}\label{boundedA}
	\|A(f)\|^2_{A^{2}_{\omega}}\, &  \lesssim  \sum_{j}\,\frac{\,| <f,e_j>|^2}{\tau^4(z_j)}\,\tau^4(z_j)
\, \lesssim  \|f\|^2_{A^2_{\omega}}\,, \quad f\in A^{2}_{\omega}\,.
\end{align*}
\end{proof}
\subsection{Schatten class properties} 
Recall that the membership of a bounded operator $T$ in the Schatten class $S^p$ of a Hilbert space depends on the $p-$summability of its singular numbers $\{s_\nu \}_{\nu}$. However, the technique we follow in the main proof needs more than that. Below we list  the arguments of the Schatten class theory that serve as building blocks. The statements are presented under certain modifications compared to how they are presented in \cite{FX} taking into account the fact that the singular values of $|T|^p\doteqdot (T^*T)^{\frac p2}$ are the $\{ s_{\nu}^p\}_{\nu}\,.$ In other words $\|T \|^p_{S^p} = \||T|^p\|_ {S^1}\,.
$ The same Lemmas were used in \cite{FX} for the case of the classical weighted Bergman spaces. 
\begin{Lemma D}
Let $p\in (0,\infty).$ Then for every pair of finite-rank operators $A$ and $B$ 
$$
\sum_{\nu \geq 1}\,(s_{\nu} (AB))^p \leq 2 \,\sum_{\nu \geq 1}\,(s_{\nu}(A))^p\,\sum_{\nu \geq 1}\,(s_{\nu} (B))^p\,\,.
$$
\end{Lemma D}	
	\begin{proof}
		See Lemma $6.7$ \cite{FX}\,.
	\end{proof}
For the product of two bounded operators we know that 
\begin{Lemma E}
	Let $A$ and $B$ be two bounded operators. Then 
	$$
	\|AB\|^p_{S^p} \leq \|B\|^p\, \|A\|^p_{S^p} \quad \text{and} \quad 
	\|BA\|^p_{S^p} \leq \|B\|^p\, \|A\|^p_{S^p}
	$$
	for every $p\in (0,1)\,.$
\end{Lemma E}	
\begin{proof}
	See Lemma $3.3$ \cite{FX}\,.
\end{proof}
The next one is the analogue of Lemma$3.4$\cite{FX} for two operators.
\begin{Lemma F}
	Let $A=A_1 + A_2$ be the sum of two finite-rank operators on a Hilbert space. 
	Then for every $\p\in (0,1)$
	$$
	\|A\|_{S^p}^p \lesssim \|A_1\|_{S^p}^p + \|A_2\|_{S^p}^p\,.
	$$
	The constant involved depends only on $p$.
	 \end{Lemma F}
 
  \section{PROOF OF THEOREM 1}
  $(1) \Rightarrow  (2) \,:\,$ 
  Assume that $H_g \in S^p$.
  Let  $ n \in \mathbb N $ such that $n > \frac{2p+4}{3p}\,.$ 
   If $\delta \in (0,m_{\tau})$  then we consider a $ (\delta, \tau)-$ lattice of $\mathbb D$ as in Lemma $A$.
According to Lemma B, choosing a positive integer $k$  we can partition the lattice 
into a finite number $M$ of infinite subsequences such that 
\begin{equation} \label{separation1}
d_{\tau}(z_m, z_{\nu})\geq 2^k\delta
\end{equation}
for any two distinct points $z_m$ and $z_\nu$ of the same subsequence.
This positive integer $k$ will be chosen appropriately large. 
It is enough to work with one of these subsequences. We denote  as $\{z_j\}_j$ 
an enumeration of its points with $|z_j|\geq \rho_0$ where $\rho_0$ is that of Lemma $C$.
Let $J$ be any finite subcollection. If $\{ e_j\}_j$ is an orthonormal set of $A^{2}_{\omega}$
we define two operators.
 One is the
 $$
 A(f)=\sum_{j\in J}\, <f,e_j>\kappa_j \,,\quad\quad f\in A^{2}_{\omega} ,
 $$
 where  $\kappa_j(z)=\frac{F_{z_j, n}(z)}{\|F_{z_j, n} \|_{A^{2}_{\omega}}},\,z\in \mathbb D\,.$
  According to Proposition \ref{operatorA} 
  \begin{align}\label{boundedA}
    \|A(f)\|^2_{A^{2}_{\omega}}
    \,& \lesssim  \|f\|^2_{A^2_{\omega}}\,, \quad f\in A^{2}_{\omega}\,.
    \end{align}
The constant in the estimation above is independent of the collection $J$\,.
 The other operator is the 
$$
B(f)= \sum_{j\in J}\,c_j\, <f,h_j>\,e_j \,,
 $$ 
 where  $\{c_j\}_{j \in J}$ are non-negative numbers and 
\begin{equation}\label{h_j}
h_j =\begin{cases} \frac{ \chi_{D_{j,\delta}}\,g\kappa_j-P_{j,\delta}(g\kappa_j)}{\| \chi_{D_{j,\delta}}\,g\kappa_j-P_{j,\delta}(g\kappa_j)\|_{L^{2}_{\omega}}} \,\,,\quad \| \chi_{D_{j,\delta}}\,g\kappa_j-P_{j,\delta}(g\kappa_j)\|_{L^{2}_{\omega}}\neq 0 \\
\\
0\quad\,\,\,,\quad  \quad \quad \text{otherwise}\end{cases}
\end{equation}
Notice that in (\ref{h_j}) we have set $D_{j,\delta}\doteqdot D(\delta\tau(z_j))$
and $P_{j,\delta} \doteqdot P_{z_j,\delta}$ in order to simplify the presentation. We will keep this notation.
Since we are talking for a large positive integer $k$ the properties (\ref{c2}), (\ref{separation1}) imply that the discs $D_{j,\delta}$ are disjoint.
As a consequence, the functions $h_j$  satisfy the properties 
$$
\|h_j\|\leq 1\,\,, \quad \quad <h_j,h_i>=0 \quad j\neq i\,\,.
$$
We  define  the  product
\begin{align}\label{product}
B H_g A(f)&=\sum_{j\in J}\,\sum_{k\in J}\,c_j\,<H_g(\kappa_i)\,,\,h_j >\,<f, e_i>\, e_j \,,\quad f \in A^2_{\omega}\,.
\end{align}
If we recall Lemma $D$ and Lemma $E$ then
\begin{align*}
\|B H_g  A\|_{S^p}^p &\lesssim \sum_{\nu} (s_{\nu}(B))^p (s_{\nu}(H_g  A))^p 
\leq \| (s_{\nu}(B))^p\|_{l^{\infty}} \|H_g A\|_{S^p}^p\\
&\leq \| (s_{\nu}(B))^p\|_{l^{\infty}}\,\, \|H_g\|_{S^p}^p \,\, \|A\|^p_{A^{2}_{\omega} \to A^{2}_{\omega}}\,.
\end{align*}
Since $\| (s_{\nu}(B))^p\|_{l^{\infty}} = \sup_{j\in J}\,c^p_j $  we get that 
\begin{align}\label{norm product}
	\|B H_g  A\|_{S^p}^p &\lesssim\,\, \sup_{j\in J}\,c^p_j \,\,.
\end{align}	
Consider the decomposition
\begin{align*}
 B H_g A(f)
=& \sum_{j\in J} \,c_j\,<H_g(\kappa_j)\,,\,h_j >\,<f, e_j >\, e_j \\
&\quad + \sum_{j,\,i\in J\,:\, j\neq i}\,c_j\,<H_g(\kappa_i)\,,\,h_j >\,<f, e_i>\, e_j = Y\, + \,Z\,\,.
\end{align*}
By Lemma $F$ 
\begin{equation}\label{FX}
\|Y\|^p_{S^p}\,\leq\,2\, \|B H_g A\|^p_{S^p}\, +\,2\,\|Z\|^p_{S^p}\,.
\end{equation}

 Since $Y$ is a positive, diagonal operator 
\begin{align*}
\|Y\|^p_{S^p}\,&= \sum_{j\in J} \,c^p_j\,|<H_g(\kappa_j)\,,\,h_j >|^p  = \sum_{j\in J} \,c^p_j\,|<g\,\kappa_j\,- \,P(g\,\kappa_j)\,,\,h_j >|^p \\
& =   \sum_{j\in J} \,c^p_j\,| <\chi_{D_{j,\delta}}g\,\kappa_j - P_{j, \delta}(g\,\kappa_j)\,,\,h_j  >|^p
\end{align*}
where in the last line we have used Lemma\ref{local bergman}. Now  we employ (\ref{h_j}),  so
\begin{align*}
\|Y\|^p_{S^p}\,& =  \sum_{j\in J} \,c^p_j\, \| \chi_{D_{j,\delta}}\,g\kappa_j-P_{j,\delta}(g\kappa_j)\|_{L^{2}_{\omega}}^p\\
&=  \sum_{j\in J} \,c^p_j\,\, \left\{ \int_{D_{j,\delta}} \,|g(z)\kappa_j(z) -P_{j,\delta}(g\kappa_j)(z)|^2\,\omega(z)\,dA(z)\,\right\}^{\frac{p}{2}}\\
&= \sum_{j\in J} \,c^p_j\,\, \left\{ \int_{D_{j,\delta}} \,|\kappa_j(z)|^2\,\omega(z)\,|g(z) -\,\frac{1}{\kappa_j(z)}\,P_{j,\delta}(g\kappa_j)(z)|^2\,dA(z)\,\right\}^{\frac{p}{2}}\\
& \asymp  \sum_{j\in J} \,c^p_j\,\, \left\{ \int_{D_{j,\delta}} \,|g(z) -\,\frac{1}{\kappa_j(z)}\,P_{j,\delta}(g\kappa_j)(z)|^2\,dA(z)\,\right\}^{\frac{p}{2}}\,.
\end{align*}
The last estimation 
 above is due to (\ref{test function1}). As a conclusion
\begin{align}\label{Y}
\|Y\|^p_{S^p}\,& \gtrsim \,\,\sum_{j\in J} \,c^p_j\,\,  F_{\delta}^p(z_j)\,.
\end{align}

Now, applying Proposition $1.29$ \cite{Zh3},
\begin{align*}
\|Z\|^p_{S^p}&\leq  \sum_{j,\,i\in J\,:\, j\neq i}\,c^p_j\,|<H_g(\kappa_i)\,,\,h_j >|^p= \sum_{j,\,i\in J\,:\, j\neq i}\,c^p_j\,|<g\,\kappa_i - P(g\,\kappa_i)\,,\,h_j   >|^p\\
&= \sum_{j,\,i\in J\,:\, j\neq i}\,c^p_j\,|<\chi_{D_{j,\delta}}g\,\kappa_i - P_{j, \delta}(g\,\kappa_i)\,,\,h_j  >|^p\\
&\leq \sum_{j,\,i\in J\,:\, j\neq i}\,c^p_j\,\left\{\int_{D_{j,\delta}}\,|g(z)\,\kappa_i(z) - P_{j, \delta}(g\,\kappa_i)(z)|^2\, \omega(z)\, dA(z)\right\}^{\frac{p}{2}}\\
&\leq \sum_{j,\,i\in J\,:\, j\neq i}\,c^p_j\,\left\{\int_{D_{j,\delta}}\,|g(z)\,\kappa_i(z) - \,\kappa_i(z)\,{\mathcal P}_{j, \delta}(g)(z)|^2\, \omega(z)\, dA(z)\right\}^{\frac{p}{2}}\,,
\end{align*}
where ${\mathcal P}_{j, \delta}$ is the projection of $L^2(D_{j,\delta})$ onto $A^2(D_{j,\delta})$. Therefore
\begin{align}
\|Z\|^p_{S^p}&\leq \sum_{j,\,i\in J\,:\, j\neq i}\,c^p_j\,\left\{\int_{D_{j,\delta}}\,|g(z)\,\kappa_i(z) - \,\kappa_i(z)\,{\mathcal P}_{j, \delta}(g)(z)|^2\, \omega(z)\, dA(z)\right\}^{\frac{p}{2}}\nonumber\\
& \leq \sum_{j,\,i\in J\,:\, j\neq i}\,c^p_j\,\sup_{z\in D_{j,\delta}}\left(|\kappa_i(z)|^p\,\omega^{\frac p2}(z)\right)
\left\{\int_{D_{j,\delta}}\,|g(z)\, - \,{\mathcal P}_{j, \delta}(g)(z)|^2\,  dA(z)\right\}^{\frac{p}{2}}\nonumber\\
& \asymp \sum_{j,\,i\in J\,:\, j\neq i}\,c^p_j\,\sup_{z\in D_{j,\delta}}\left(\frac{|F_{z_i,n}(z)|^p}{\tau^p(z_i)}\,\omega^{\frac p2}(z)\right)\,\tau^p(z_j)\,F^{p}_{\delta}(z_j)\nonumber\\
&=\sum_{j\in J} \, c^p_j\, \left\{ \sum_{\,i\in J\,:\, i\neq j}\, \frac{1}{\tau^p(z_i)}\,\sup_{z\in D_{j,\delta}}\,\left(|F_{z_i,n}(z)|^p\,\omega^{\frac p2}(z)\right)\right\}\,\tau^p(z_j)\,F^{p}_{\delta}(z_j)\,.\label{Z1}
\end{align}
 Let a $j \in J$ then for any $z_i$ with $i\in J,\,  i\neq j$, 
there is a unique $z^*_{j,i} \in \overline{D_{j,\delta}}$ such that 
$
\inf_{z \in \overline{D_{j,\delta}}}\, |z-z_i| = |z^*_{j,i} -z_i |\,.
$
 Moreover, by  (\ref{comparison tau})\,\,$
\tau(z)\asymp\tau(z_j) \,,\,\, \forall z\in \overline{D_{j,\delta}}\,.
$
According to  (\ref{test function2}), 
\begin{align*}
|F_{z_i,n}(z)|^p\,\omega^{\frac p2}(z) & \lesssim \,\left( \frac{\min(\tau(z)\,\tau(z_i))}{|z-z_i|}\right)^{\frac {3np}{2}}\,
|F_{z_i,n}(z)|^{\frac p2}\,\omega^{\frac p4}(z)  \\
 & \lesssim \, \left(\frac{1}{d_{\tau}( z^*_{j,i}\,,\, z_i)}\right)^{\frac {3np}{2}}\,|F_{z_i,n}(z)|^{\frac p2}\,\omega^{\frac p4}(z)\,,\quad \forall z\in D_{j,\delta}\,.
 \end{align*}
 From   (\ref{separation1})  we get that 
 $
 d_{\tau}( z^*_{j,i}\,,\, z_i) \geq 2^{k-2} \delta\,,
 $
 therefore
 \begin{align}\label{F}
 |F_{z_i,n}(z)|^p\,\omega^{\frac p2}(z) & \lesssim  \, \left( \frac 12 \right)^{\frac{3npk}{2}}\, |F_{z_i,n}(z)|^{\frac p2}\,\omega^{\frac p4}(z)
 \,,\quad \forall z\in D_{j,\delta}\,.
  \end{align}
  Combining (\ref{Z1}) and (\ref{F}) it  turns out that 
\begin{align}\label{Zfinal}
\|Z\|^p_{S^p}
& \lesssim \, \left( \frac 12 \right)^{\frac{3npk}{2}}\, \sum_{j\in J}\,c^p_j\, \left\{ \sum_{i\in J\,:\, i\neq j} \,\,\frac{1}{\tau^p(z_i)}\, \sup_{z\in D_{j,\delta}}|F_{z_i,n}(z)|^{\frac p2}\,\omega^{\frac p4}(z) \right\}\,\tau^p(z_j)\,F^{p}_{\delta}(z_j)\,.
\end{align}
Now, for each $j \in J$
and for $\delta_0= 3\delta$
 \begin{align*}
  \sum_{i\in J\,:\, i\neq j} &\,\,\frac{1}{\tau^p(z_i)}\, \sup_{z\in D_{j,\delta}}|F_{z_i,n}(z)|^{\frac p2}\,\omega^{\frac p4}(z)
 \leq\sum_{i:\, i\neq j, |z_i| \geq \rho_0}\,\,\frac{1}{\tau^p(z_i)}\, \sup_{z\in D_{j,\delta}}|F_{z_i,n}(z)|^{\frac p2}\,\omega^{\frac p4}(z)\\
&= \sum_{\left\{z_i : |z_j-z_i|\leq \delta_0 \tau (z_j)\right\}}\,\,\frac{1}{\tau^p(z_i)}\,
\,\sup_{z\in D_{j,\delta}}|F_{z_i,n}(z)|^{\frac p2}\,\omega^{\frac p4}(z)\\
&\quad\quad +\sum_{l\geq 0}\sum_{\left\{z_i : 2^{l}\delta_0 \tau(z_j)<|z_j-z_i|\leq 2^{l+1}\delta_0 \tau (z_j)\right\}}
\,\,\frac{1}{\tau^p(z_i)}\,
\,\sup_{z\in D_{j,\delta}}|F_{z_i,n}(z)|^{\frac p2}\,\omega^{\frac p4}(z)\,.
\end{align*}
As far as it concerns the first additive, due to (\ref{test function2}) it holds that 
$$|F_{z_i,n}(z)|^{\frac p2}\,\omega^{\frac p4}(z) \lesssim 1\,, \quad z\in D_{j,\delta},$$
 so we get
\begin{align*}
&\sum_{\left\{z_i : |z_j-z_i|\leq \delta_0 \tau (z_j)\right\}}\,\,\frac{1}{\tau^p(z_i)}\,
\,\sup_{z\in D_{j,\delta}}|F_{z_i,n}(z)|^{\frac p2}\,\omega^{\frac p4}(z)
&\lesssim \, \sum_{\left\{z_i : |z_j-z_i|\leq \delta_0 \tau (z_j)\right\}}\,\,\frac{1}{\tau^p(z_i)}
\lesssim\,\,\frac{1}{\tau^p(z_j)}
\end{align*}
where in the last inequality we have employed (\ref{comparison tau}) and the finite covering property of the $(\delta,\tau)-$lattice.
The second additive is estimated as 
\begin{align*}
&\sum_{l\geq 0}\sum_{\left\{z_i : 2^{l}\delta_0 \tau(z_j)<|z_j-z_i|\leq 2^{l+1}\delta_0 \tau (z_j)\right\}}
\,\,\frac{1}{\tau^p(z_i)}\,
\,\sup_{z\in D_{j,\delta}} \,|F_{z_i,n}(z)|^{\frac p2}\,\omega^{\frac p4}(z) \\
&\lesssim  \sum_{l\geq 0}\sum_{\left\{z_i : 2^{l}\delta_0 \tau(z_j)<|z_j-z_i|\leq 2^{l+1}\delta_0 \tau (z_j)\right\}} 
\,\frac{ \tau^{(\frac{3n}{2}-1)p}(z_i)}{|z^*_{j,i} - z_i|^{\frac{3np}{2}}}\\
& \lesssim \, \delta_0^{-\frac{3np}{2}}\, \tau^{-\frac{3np}{2}}(z_j)\,\sum_{l\geq 0}\, \left(\frac 12\right)^{\frac{3npl}{2}}\,\sum_{\left\{z_i : 2^{l}\delta_0 \tau(z_j)<|z_j-z_i|\leq 2^{l+1}\delta_0 \tau(z_j)\right\}}  \tau^{(\frac{3n}{2}-1)p}(z_i)\,.
\end{align*}
Since $n \geq \frac {4+2p}{3p}$ 
\begin{align*}
\sum_{\left\{z_i : 2^{l}\delta_0 \tau(z_j)<|z_j-z_i|\leq 2^{l+1}\delta_0 \tau(z_j)\right\}} & \tau^{(\frac{3n}{2}-1)p}(z_i)
 \leq \left(\sum_{\left\{z_i : 2^{l}\delta_0 \tau(z_j)<|z_j-z_i|\leq 2^{l+1}\delta_0 \tau(z_j)\right\}}  \tau^{2}(z_i)\right)^{\frac{(3n-2)p}{4}}\\
&\lesssim\left(A(D_{z_j, 2^{l+1}\delta_0 })\right)^{\frac{(3n-2)p}{4}} \lesssim\, \left(2^{2l}\,\tau^2(z_j)\right)^{\frac{(3n-2)p}{4}}
\end{align*}
Thus
\begin{align*}
\sum_{l\geq 0}&\sum_{\left\{z_i : 2^{l}\delta_0 \tau(z_j)<|z_j-z_i|\leq 2^{l+1}\delta_0 \tau (z_j)\right\}}
\,\,\frac{1}{\tau^p(z_i)}\,
\,\sup_{z\in D_{j,\delta}} \,|F_{z_i,n}(z)|^{\frac p2}\,\omega^{\frac p4}(z) \\
& \lesssim \, \delta_0^{-\frac{3np}{2}}\, \tau^{-\frac{3np}{2}}(z_j)\, \tau^{\frac{(3n-2)p}{2}}(z_j)\,\sum_{l\geq 1}\,\,\left(\frac 12\right)^{pl}
 \lesssim \,\, \tau^{-p}(z_j)\,.
\end{align*}
So
\begin{align*}
 \sum_{i\in J\,:\, i\neq j} &\,\,\frac{1}{\tau^p(z_i)}\, \sup_{z\in D_{j,\delta}} |F_{z_i,n}(z)|^{\frac p2}\,\omega^{\frac p4}(z) 
 \lesssim\,  \tau^{-p}(z_j)\,.
\end{align*}
 Combining the above estimations with (\ref{Zfinal}) we get that
  \begin{align}\label{Z}
  \|Z\|^p_{S^p}&\lesssim \,\left(\frac 12\right)^{\frac {3npk}{2}}\,\sum_{j\in J}\,c^p_j\, F^{p}_{\delta}(z_j)\,.
  \end{align}
    Choosing $k$ appropriately large and taking into account  (\ref{norm product}),  (\ref{FX}), (\ref{Y}), (\ref{Z})    we conclude that 
   \begin{align*}
   \sum_{j\in J} \,c^p_j\,\,  F_{\delta}^p(z_j)
   \lesssim  \sup_{j\in J} \,c^p_j\,,
   \end{align*}
    for every collection $J.$  By duality we get the desired result.
  
     $(2) \Rightarrow  (3) :$ We think as in Theorem $4.4'$ \cite{LR2}. The only difference 
     is that here $p\in (0,1$) and $\omega \in \mathcal W$.  
The boundedness of $H_g$ implies that the symbol $g$ admits a decomposition $g=g_1+g_2$ and the additives 
satisfy the condicions of case $(3)$ in Theorem A. As in Theorem $4.4'$ we recall that in the proof of the implication $(2) \Rightarrow (3)$ of Theorem $4.1$ in \cite{LR1}
the authors conclude that 
\begin{align*}
\frac{1}{|D(\frac{\delta}{5}\tau(z))|}&\,\int_{D\left(\frac{\delta}{5}\tau(z)\right)} |g_1(u)|^2\,dA(u)\\
 &\lesssim \sup\left\{ F^2_{\frac{\delta}{5}}(w) : w\in D\left(\frac{3\delta}{5}\tau(z)\right) \right\}
  \lesssim F^2_{\frac{4\delta}{5}}(z) \lesssim F^2_{\delta}(z)
\end{align*}
and 
\begin{align*}
\left|\frac{\bar{\partial}g_2(z)}{(\Delta \phi(z))^{\frac 12}}\right|
 &\lesssim \sup\left\{ F_{\frac{\delta}{5}}(w) : w\in D\left(\frac{3\delta}{5}\tau(z)\right) \right\}
 \lesssim F_{\frac{4\delta}{5}}(z) \lesssim F_{\delta}(z)\,,
\end{align*}
 for every $z\in \mathbb D$. 
Therefore, if $\{z_j\}_{j}$ is a $(\frac{\delta}{5},\tau)-$ lattice, it holds that
\begin{align*}
\sum_{j}  &\left\{\frac{1}{|D(\frac{\delta}{5}\tau(z_j))|}\,\int_{D\left(\frac{\delta}{5}\tau(z_j)\right)} |g_1(z)|^2\,dA(z) \right\}^{\frac p2}
 \lesssim \sum_j F^{p}_{\delta}(z_j)
\end{align*}
and 
\begin{align*}
\sum_{j}  &\left\{\frac{1}{|D(\frac{\delta}{5}\tau(z_j))|}\,\int_{D\left(\frac{\delta}{5}\tau(z_j)\right)} \left|\frac{\bar{\partial}g_2(z)}{(\Delta \phi(z))^{\frac 12}}\right|\,dA(z) \right\}^{\frac p2}
\lesssim \sum_j F^{p}_{\delta}(z_j)\,.
\end{align*}
\\
Under our assumption the sums on the right are finite.

$(3) \Rightarrow (1)\,:\,$ 
Although $p\in (0,1)$ and $\omega \in \mathcal W$ the proof of this implication is also based on the proof of the corresponding step of Theorem $4.4'$ \cite{LR2}.  Assume that the symbol $g$ is decomposed as in $(3)$. 
 If we recall the boundedness of the projection operator then it turns out  that
    $$
    \|H_{g_1}(f)\|_{ L^2_{\omega}}\, \leq \|g_1 f \|_{ L^2_{\omega}} + \|P(g_1 f)\|_{ L^2_{\omega}}\lesssim \,\|M_{g_1}(f)\|_{ L^2_{\omega}}
   $$
   where $M_{g_1}(f)=g_1 f$ is the multiplication operator with symbol the function $g_1$.
   On the other hand $H_{g_2}(f)$ is the solution of the $\bar{\partial}u=\bar{\partial}g_2\,f$
   with minimal $L^2_{\omega}$ norm. So we can apply Lemma $4.1.1$ \cite{H} according to which there is a solution $u$ such that 
   $$ 
   \int_{\mathbb D}\, |u(z)|^2\, \omega(z)\, dA(z)\lesssim
   \int_{\mathbb D}\, |\bar{\partial}g_2(z)\,f(z)|^2\, \frac{\omega(z)}{\Delta(\phi(z))}\, dA(z)
   $$
    and get that 
    $$\|H_{g_2}(f)\|_{ L^2_{\omega}}\leq \|u\|_{ L^2_{\omega}} \lesssim \,\|M_{\frac{\bar{\partial}g_2}{(\Delta \phi)^{\frac 12}}}(f)\|_{ L^2_{\omega}}\,.$$
     Since we are dealing with a bounded $H_g$ the conditions (\ref{CM}), (\ref{bar}) hold which means that these multiplication operators are bounded.    
    Due to the inequalities above, if 
    $M_{g_1} $ and $M_{\frac{\bar{\partial}g_2}{(\Delta \phi)^{\frac 12}}}$ belong to $S^p$ then  $H_{g_1}$ and $H_{g_2}$ belong to  $S^p$ too.
        Now,  by a standard argument,   
    $$
    M^{*}_{\psi}M_{\psi}(f)= P(|\psi|^2 f) \doteqdot T_{|\psi|^2}(f)
    $$
    where $T_{|\psi|^2}$  stands for the Toeplitz operator with symbol $|\psi|^2$.
    As a consequence of this identity we conclude that  
    $$
    M_{\psi}\in S^p \Leftrightarrow T_{|\psi|^2}\in S^{\frac p2}\,.
    $$
    Here $\psi=g_1\,$ or \,\,$\frac{\bar{\partial}g_2}{(\Delta \phi)^{\frac 12}}$\,\,\,. 
    At this point, since $p\in (0,1)$ we have to use Theorem $4.6$ of \cite{APP} according to which the conditions of $(3)$ on $g_1, g_2$ are sufficient in order
    $
    T_{|g_1|^2} \quad \text{and} \quad T_{\left|\frac{\bar{\partial}g_2}{(\Delta \phi)^{\frac 12}}\right|^2}
    $
    to be in $S^{\frac p2}$, \,$p\in(0,1)$. In that case 
    $
    M_{g_1} \, \text{and} \, M_{\frac{\bar{\partial}g_2}{(\Delta \phi)^{\frac 12}}}
    $
    are in $S^p$, hence $H_g=H_{g_1}+H_{g_2}\in S^p$. 

\end{document}